\documentclass[12pt]{article}
\usepackage{graphicx}
\usepackage{amsmath}
\usepackage{amsfonts}
\usepackage{amsthm}
\usepackage{verbatim}
\usepackage[pdftex]{hyperref}
\usepackage{color}

\definecolor{orange}{RGB}{242, 77, 16}
\definecolor{terquoise}{RGB}{0, 255, 255}

\usepackage[T1]{fontenc}
\usepackage[utf8]{inputenc}

\newtheorem{theorem}{Theorem}[section]
\newtheorem*{theorem*}{Theorem}
\newtheorem{lemma}[theorem]{Lemma}
\newtheorem{proposition}[theorem]{Proposition}
\newtheorem{corollary}[theorem]{Corollary}

\theoremstyle{definition}

\newtheorem*{problem*}{Problem}

\begin{document}

\title{Chip-Firing and Fractional Bases}

\author{Matvey Borodin \and Hannah Han \and Kaylee Ji \and Tanya Khovanova \and Alexander Peng \and David Sun \and Isabel Tu \and Jason Yang \and William Yang \and Kevin Zhang \and Kevin Zhao}
\date{}

\maketitle

\begin{abstract}
We study a particular chip-firing process on the one-dimensional lattice. At any time when there are at least $a+b$ chips at a vertex, $a$ chips fire to the left and $b$ chips fire to the right. We describe the final state of this process when we start with $n$ chips at the origin.
\end{abstract}

\section{Introduction}

A chip-firing game is a one-player game played on a graph. Chips are distributed on vertices of a graph. At each point, a vertex that has the number of chips that is at least its degree can `fire' by simultaneously sending a chip along each incident edge (see \cite{BLS}). The game was later expanded to directed graphs, where a move consists of selecting a vertex with at least as many chips as
its outdegree, and sending one chip along each outgoing edge to its neighbors (see \cite{BL}). 

We study a particular chip-firing game suggested by Prof.~James Propp. This chip-firing game is played on an integer line. We consider an integer line as a graph with integers as vertices, and neighboring integers are connected by edges. We position the line horizontally, so that the left and right neighbors are well-defined. At any point in the game each vertex has non-negative number of chips assigned to it.

Fixing positive integers $a$ and $b$, we now describe what we call an $a$-$b$ chip-firing game. If there are at least $a+b$ chips at a particular vertex, then the vertex `fires': $a$ chips are moved to the left neighbor and $b$ chips are moved to the right neighbor. This game can be considered as a game on a directed graph if we replace an edge connecting two consecutive integers with $a$ edges directed from the larger to the smaller integer and $b$ edges directed the opposite way. Without loss of generality we assume throughout the paper that $a \leq b$.

We start with preliminaries and examples in Section~\ref{sec:preliminaries}, also describing fractional bases as they are related to this game. In Section~\ref{sec:series} we represent a state in the game as a number and a Laurent polynomial. We discuss some properties of the state and divide the state into the left and the right part. The left part corresponds to the vertices with non-positive indices, and the right part to the vertices with positive integers.

In Section~\ref{sec:notcoprime} we reduce the case of not coprime $a$ and $b$ to the case of their coprime factors. In Section~\ref{sec:a-a} we describe the final state for the case when $a = b$.

In Section~\ref{sec:23firing} we consider an example of 2-3 firing which illustrates what happens for any $a$-$b$ firing. In the next two sections we do some preliminary work: In Section~\ref{sec:originfiring} we describe what happens to the right part when the origin fires once. In Section~\ref{sec:settlements} we define and describe settlements, that is, the sequence of intermediate states for the right part.

We give a description of the final state in the next three sections. In Section~\ref{sec:value} we prove that the right part of the final state when evaluated in base $\frac{b}{a}$ becomes a constant for large initial number of chips. In Section~\ref{sec:leftside} we prove that for large initial number of chips the transition of the final state from $n$ to $n+1$ chips can be described using base $\frac{b}{a}$. In Section~\ref{sec:rightside} we show how to use settlements to calculate the righ part of the final state of the game.

We finish the paper with an example of 1-$b$ firing in Section~\ref{sec:1bfiring}.

\section{Preliminaries and Examples}\label{sec:preliminaries}

Our graph is an infinite path indexed by integers. We visualize it as a horizontal line with integers increasing to the right.

We start with $n$ chips at the vertex marked 0. We call the starting vertex \textit{the origin}. Whenever a vertex $x$ has at least $a+b$ chips, it fires. Firing means sending $a$ chips to the left and $b$ chips to the right. We call our game an \textit{$a$-$b$ chip-firing game}.

In most chip-firing models one chip is sent along the edge. This chip-firing model can be represented in the standard way by replacing each edge with $a$ edges directed to the left and $b$ edges directed to the right. By reflecting the line with respect to its origin we can swap $a$ and $b$. Without loss of generality we assume throughout the paper that $a \leq b$.

Bj\"{o}rner, Lov\'{a}sz, and Shor \cite{BLS} showed that a chip-firing game is finite if the number of chips is less than the number of edges. Thus, our game is finite. The order in which vertices fire does not matter, which means that the game terminates at a unique state, which we call the \textit{final state}, (see \cite{BLS}). 

We represent each configuration with a string of integers, where we drop all zeros on the left and on the right and where we mark the origin with a dot to its right. We call this dot \textit{the radix}. For example, consider 1-2 chip-firing. If we start with seven chips at the origin, we get 22.12 as the final state, see Figure~\ref{fig:cf7}.

\begin{figure}[htp]
\centering
\includegraphics[scale=0.35]{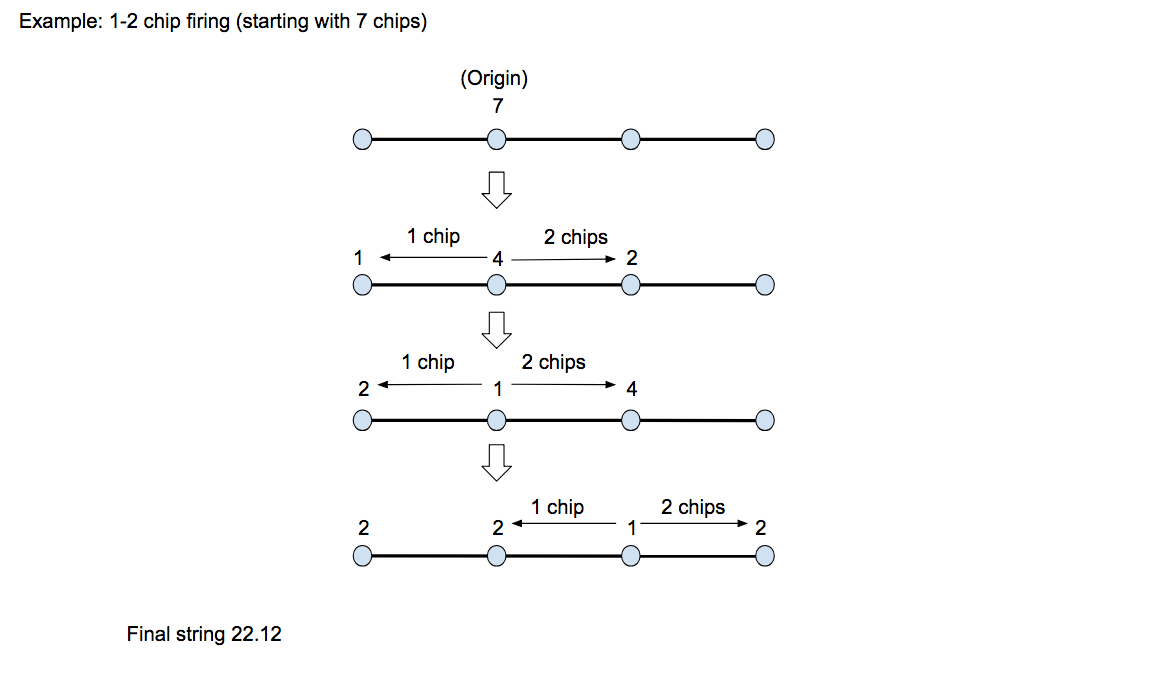}
\caption{1-2 Chip-firing starting with 7 chips}\label{fig:cf7}
\end{figure}

Given the final string representing an integer $n$, the final string for $n+1$ can be found by simply increasing the number of chips at the origin by 1, then firing if necessary. For example, the final state for $n=7$ in the 1-2 chip-firing game is 22.12. Thus, to find the final state for $n= 8$ in the 1-2 chip-firing game, we add 1 to 22.12 resulting in 23.12. Then, after the necessary firings the final state for $n= 8$ is 111.1112.

Our results will connect the left part of the final state to the fractional base $\frac{b}{a}$, where $ b > a$. These fractional bases were suggested by Propp \cite{JP} and described as a division algorithm in \cite{AFS}. This algorithm was also called \textit{exploding dots} and was popularized by Tanton \cite{JT}. 

We denote the base $\frac{b}{a}$ representation of $n$ as $(n)_\frac{b}{a}$. Conversely, given a string of digits $w$ in base $\frac{b}{a}$, its evaluation is denoted as $[w]_\frac{b}{a}$.

Here is a recursive description of the integer $n$ in such a base, where
\[(n)_\frac{b}{a} = d_kd_{k-1}\ldots d_2d_1d_0.\]
The last digit $d_0$ is the remainder of $n$ modulo $b$. The rest of the digits, $d_k d_{k-1}\cdots d_1$, is $\left(\frac{a(n-d_0)}{b}\right)_\frac{b}{a}$. The number $n$ can be recovered from its base $\frac{b}{a}$ representation as (see \cite{JP})
\[n = [d_k d_{k-1}\cdots d_0]_\frac{b}{a} = \sum_{i=0}^k d_i\frac{b^i}{a^i}.\]

Informally, we can explain this representation as follows. We start with $n$ in the units' place. If $n \geq b$ we subtract $b$ from $n$ and add $a$ to the place to the left. The same rule works independently of placement and continues until all the digits are below $b$. We can view this as an exploding dots process, where chips are dots, and $b$ dots at one vertex explode to be replaced with $a$ dots at the vertex to the left, as popularized by Tanton \cite{JT}. 

For example, to calculate 5 in base $\frac{3}{2}$, we start with 5 dots in the rightmost box, box 0. We can represent this state as 5. Since we have more than three dots we have an explosion: the number of dots in the rightmost box decreases by 3 and we add 2 dots to the box on the left. As the result we get $(5)_\frac{3}{2}= 22$. This example is represented in Figure~\ref{fig:base3over2}.

\begin{figure}[ht!]
\centering
\includegraphics[scale=0.4]{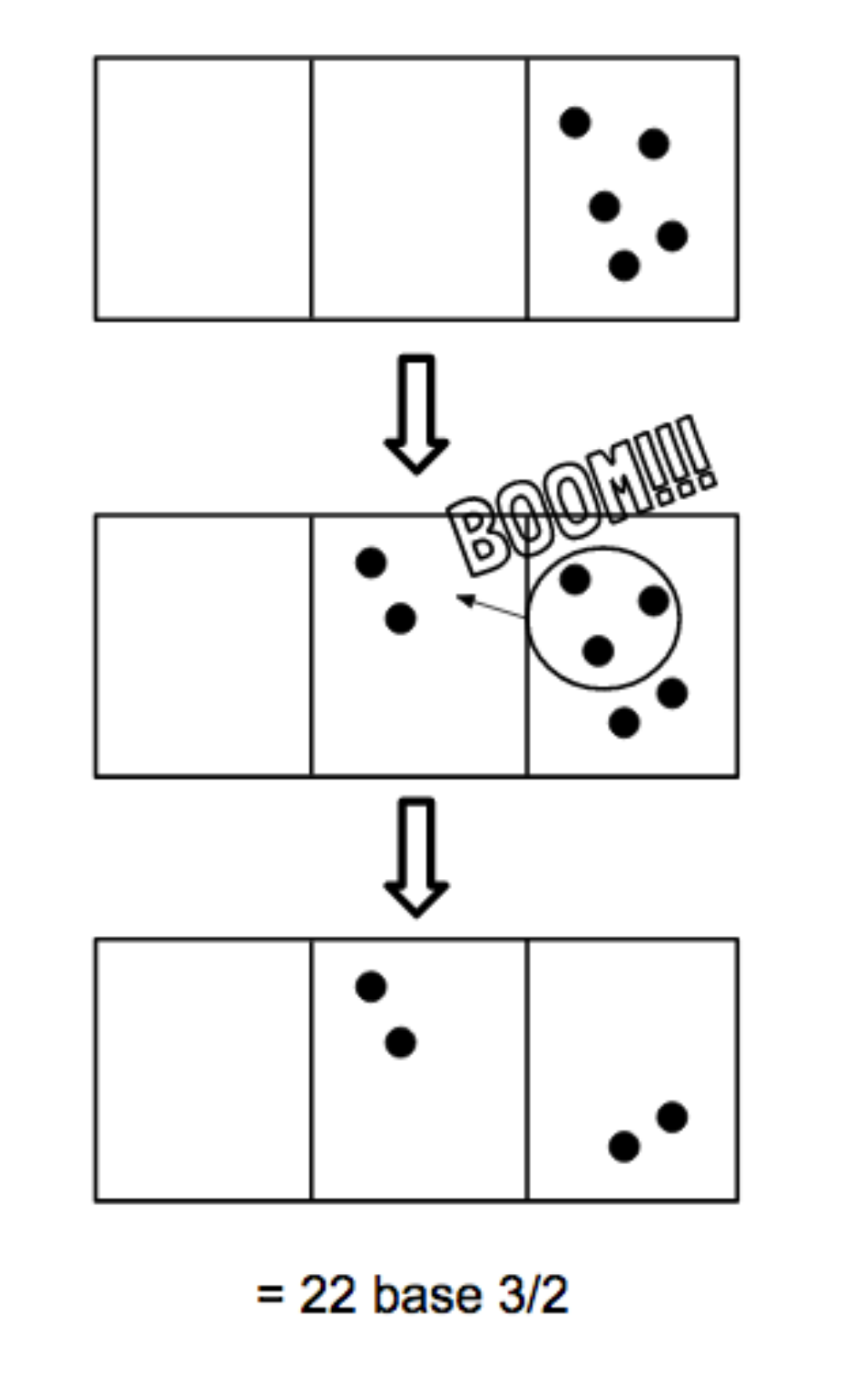}
\caption{Exploding dots show how to represent 5 in base $\frac{3}{2}$.}\label{fig:base3over2}
\end{figure}

We can also see such bases as a chip-firing process on a graph with vertices indexed by non-negative integers and an additional vertex called the \textit{sink}. Vertex $i$ has $a$ outgoing edges to vertex $i-1$ and $b-a$ outgoing edges to the sink. Placing $n$ chips at the vertex marked 0, the chip-firing process ends in a state where there are $d_i$ chips at vertex $-i$ for each $i$, where $(n)_\frac{b}{a} = d_kd_{k-1}\ldots d_2d_1d_0$. This approach was introduced by Propp \cite{JP}.

The first few non-negative integers written in base $\frac{3}{2}$ form sequence A024629 in the OEIS \cite{OEIS}: 
\[0, 1, 2, 20, 21, 22, 210, 211, 212, 2100, \ldots.\]

Similarly to integers, we can express a real number in base $\frac{b}{a}$ by allowing digits after the radix:
\[[d_k d_{k-1}\cdots d_0.d_{-1}d_{-2}\ldots]_\frac{b}{a} = \sum_{i=\infty}^k d_i\frac{b^i}{a^i}.\]
Such bases were thoroughly studied by Akiyama et al.~\cite{AFS} and by Frougny and Klouda \cite{FK}. For a rational base $\frac{b}{a}$ this representation allows using digits $\{0,1,\ldots,b-1\}$. The representation of a number in this base is not unique, but integers can be uniquely represented as finite strings.

\section{The state of the game}\label{sec:series}

Our game starts in the state at which all $n$ chips are at the origin. After several firings we will reach a state $\sigma = s_{-k}\ldots s_{-2}s_{-1}s_0.s_1s_2\ldots s_\ell$ for some integers $k$ and $\ell$, where $s_{-k} > 0$ and $s_{\ell} > 0$. The state $\sigma$ means there are $s_i$ chips at vertex $i$ for each $i$. Firing from a vertex $i$ means replacing $s_i$ with $s_i - a-b$, replacing $s_{i-1}$ with $s_{i-1}+a$ and replacing $s_{i+1}$ with $s_{i+1}+b$. If $\sigma$ is the final state, then $s_i < a + b$ for $-k \leq i \leq \ell$.

Define the \textit{state polynomial} for state $\sigma$ as the Laurent polynomial
\[S_\sigma(t) = \sum_{m=-k}^\ell s_mt^{-m}.\] 
Note that we use negative powers of $t$ for digits located to the right of the radix. The reason for this is that we want $S_\sigma(t)$ to be equal to the evaluation of $\sigma$ in base $t$.
\[S_\sigma(t) = (\sigma)_{t}.\]

The following proposition shows that the state polynomial evaluated at $t=1$ or $t = \frac{b}{a}$ is invariant under the chip-firing procedure.

\begin{proposition}\label{prop:invariant}
For any state $\sigma$ we have
\[S_\sigma(1) = \sum_{m=-k}^\ell s_m = n\]
and 
\[S_\sigma\left(\frac{b}{a} \right) = \sum_{m=-k}^\ell s_m\left(\frac{b}{a} \right)^m= n.\]
\end{proposition}

\begin{proof}
Suppose $\tau$ is a state in the game and state $\tau'$ is the state reached from $\tau$ by firing vertex $i$. Then 
\[S_{\tau'}(t) = S_\tau(t) + at^{-i+1} - (a+b)t^{-i} + bt^{-i-1} = S_\tau(t) + t^{-i-1}(at -b)(t -1).\]
Thus, $S_{\tau'}(1) = S_{\tau}(1)$ and $S_{\tau'}(\frac{b}{a}) = S_{\tau}(\frac{b}{a})$. Repeated application of this argument shows that if $\tau''$ is any state reachable from $\tau$ by a sequence of vertex-firings, then $S_{\tau''}(1) = S_{\tau}(1)$ and $S_{\tau''}(\frac{b}{a}) = S_{\tau}(\frac{b}{a})$. The result now follows by setting $\tau$ equal to the initial state and $\tau'' = \sigma$, and then noting that for the initial state $\tau$ we have 
\[S_{\tau}(1) = S_{\tau}\left(\frac{b}{a} \right) = n.\]
\end{proof}

As we will see soon, for $a \neq b$ the left parts of states look different from the right parts. Given state $\sigma = s_{-k}\ldots s_{-2}s_{-1}s_0.s_1s_2\ldots s_\ell$, we call $s_{-k}\ldots s_{-2}s_{-1}s_0$ its \textit{left part} and denote it as $\sigma^L$. Similarly, we call $0.s_1s_2\ldots s_\ell$ the \textit{right part} of the state $\sigma$ and denote it as $\sigma^R$. Correspondingly, we can divide the polynomial $S_\sigma(t)$ into its left part $S_{\sigma^L}$ and right part $S_{\sigma^R}$, where
\[S_\sigma(t) = S_{\sigma^L}(t) + S_{\sigma^R}(t),\]
where $S_{\sigma^L}(t) = \sum_{m=-k}^0 s_mt^{-m}$ and  $S_{\sigma^R}(t) = \sum_{m=1}^\ell s_mt^{-m}$.

Before proving the following proposition we introduce a name for the vertex to the right of the origin, that is, the vertex with index 1. We call it the \textit{origout}. We also denote by $f_i(\sigma)$ the number of times the vertex $i$ fires during a game that leads from the initial state to $\sigma$.

\begin{proposition}\label{prop:valueAndFiring}
If we start with $n$ chips at the origin and reach state $\sigma$, then we have
\[S_{\sigma^L}(1) = n - bf_0(\sigma) + af_1(\sigma) \quad \textrm{ and } \quad S_{\sigma^R}(1) = bf_0(\sigma) - af_1(\sigma),\]
and 
\[S_{\sigma^L}\left(\frac{b}{a} \right)= n - a(f_0(\sigma) - f_1(\sigma)) \quad  \textrm{ and } \quad  S_{\sigma^R}\left(\frac{b}{a} \right) = a(f_0(\sigma) - f_1(\sigma)).\]
\end{proposition}

\begin{proof}
The value of $S_{\sigma^L}(1)$ or $S_{\sigma^L}(\frac{b}{a})$ can only change during the firing at the origin or origout. If the origin fires, the value of $S_{\sigma^L}(1)$ decreases by $b$ and the value of $S_{\sigma^L}(\frac{b}{a})$ decreases by $(a+b) - a\cdot \frac{b}{a} = a$. If the origout fires the value of $S_{\sigma^L}(1)$, as well as the value of $S_{\sigma^L}(\frac{b}{a})$, increases by $a$. A similar argument works for the right part, or Proposition~\ref{prop:invariant} can be used.
\end{proof}

The following consequence of the previous proposition is of importance.

\begin{corollary}
If we start with $n$ chips at the origin and reach state $\sigma$, then the values $S_{\sigma^L}(\frac{b}{a})$ and $S_{\sigma^R}(\frac{b}{a})$ are integers.
\end{corollary}

Another consequence will play a role later. It follows from the fact that the two equations involving $f_0(\sigma)$ and $f_1(\sigma)$ are linearly independent.

\begin{corollary}\label{cor:uniequelydefined}
If $a \neq b$, then $f_0(\sigma)$ and $f_1(\sigma)$ are uniquely defined, that is, they are independent of the order of the firings.
\end{corollary}

For any state $\sigma$, consider the value
\[M_\sigma = \sum_{m=-k}^\ell ms_m.\] 

In the next proposition we show how $M_\sigma$ relates to the number of firings in case $a \neq b$.

\begin{proposition}
Assume $a \neq b$. For any state $\sigma$, the number of firings that lead from the initial state with $n$ chips at the origin to state $\sigma$ is
\[\frac{M_\sigma}{a-b}.\]
\end{proposition}

\begin{proof}
Suppose $\tau$ is a state in the game and state $\tau'$ is the state reached from $\tau$ by firing vertex $i$. Then 
\[M_{\tau'} = M_\tau + a(1-i) - (a+b)(-i) + b(-i-1) = M_\tau + a-b.\]
Note that if $\tau$ is the initial state, then $M_\tau=0$. As each firing increasing this value by $a-b$, we see that $M_\sigma$ equals the number of firing times $a-b$.
\end{proof}

Our goal is to describe the final state of the $a$-$b$ chip-firing game. We denote this state as $\phi(n)$, and its left and right parts as $\phi(n)^L$ and $\phi(n)^R$ correspondingly.

\section{$a$-$b$ Firing game, when $a$ and $b$ are not coprime}\label{sec:notcoprime}

Suppose $a$ and $b$ have their greatest common divisor $d > 1$. Consider an $a$-$b$ chip-firing game starting with $n$ chips. Suppose $n = pd + q$, where $0 \leq q < d$.  We can express this game through a $\frac{a}{d}$-$\frac{b}{d}$ firing game with $p$ starting chips.

\begin{proposition}\label{prop:coprime}
Under the above assumptions, both games allow the same sequences of vertices to fire. 
Consider a state $\sigma = s_{-k}\ldots s_{-2}s_{-1}s_0.s_1s_2\ldots s_\ell$ in a $\frac{a}{d}$-$\frac{b}{d}$ firing game with $p$ starting chips that happens after firings corresponding to some sequence of vertices. Then in the $a$-$b$ chip-firing game starting with $n$ chips the same sequence of firings produces the state $\sigma' = s_{-k}'\ldots s_{-2}'s_{-1}'s_0'.s_1's_2'\ldots s_\ell'$, where $s_0' = ds_0+q$ and $s_i' = ds_i$, for $m \leq i < 0$ and $0 < i \leq \ell$.
\end{proposition}

\begin{proof}
We can group chips into bags of $d$ chips each and consider the bags to be new chips.
\end{proof}

The proposition allows to express the final state of $a$-$b$ firing through the final state of $\frac{a}{d}$-$\frac{b}{d}$ firing.

For example, the final state of 5 in a 1-1 firing is 111.11, while the final state of 25 in a 5-5 firing is 555.55 and the final state of 26 in a 5-5 firing is 556.55.

For another example, the final state of 21 in a 2-3 firing is 442.2243, while the final state of 42 in a 4-6 firing is 884.4486.

\section{An example: $a$-$a$ firing game}\label{sec:a-a}

In this section we describe the final state of the $a$-$a$ firing game. As we proved in Proposition~\ref{prop:coprime}, it is enough to describe the final state of the 1-1 firing game. As firing is symmetrical with respect to the origin, the final result is symmetrical too.  We can calculate the final state for a small initial number of chips by hand. These final states produce the following sequence:
\[1.\ 10.1\ 11.1\ 110.11\ 111.11\ 1110.111.\]

We see that a pattern emerges. To prove it we introduce the following notation. We denote the word $dddd\ldots ddd$ with $k$ $d$'s as $d_k$. In the following proof $d=1$. We also denote the word $202020\ldots 020202$ of $k$ twos alternating with zeros as $\beta_k$.

\begin{lemma}
In the 1-1 chip-firing game the final state is either 
\[1_k0.1_k \textrm{ for } n=2k \textrm{ or } 1_{k+1}.1_k \textrm{ for } n = 2k+1.\]
\end{lemma}

\begin{proof}
We prove this by induction. We already calculated small examples demonstrating the induction base. Suppose the theorem is true for $n=2k$ chips, that is, the final state for this number of chips is $1_k0.1_k$. In what follows we omit the dot, as the dot can be recovered from symmetry. Thus, the dotless final state is $1_k01_k$. If we add one more chip, the final state for $n=2k+1$ is $1_{2k+1}$. This proves the first step of the induction.

Now we add one more chip to get a state $1_k21_k = 1_k\beta_1 1_k$ which corresponds to $2k+2$ initial chips. After one firing at the origin, we get $1_{k-1}2021_{k-1}= 1_{k-1}\beta_21_{k-1}$. Now we fire from both vertices with twos to get  $1_{k-2}202021_{k-2} = 1_{k-2}\beta_31_{k-2}$. After several iterations we get the state $1_{k-i}\beta_{i+1}1_{k-i}$, after which we get to $\beta_{k+1}$. Then all twos fire again to get $10\beta_{k}01$, and again to get to $110\beta_{k-1}011$. After several iterations we get $1_i0\beta_{k-i+1}01_i$. This leads us to $1_k0201_k$, and finally to $1_{k+1}01_{k+1}$. 
\end{proof}

Applying Proposition~\ref{prop:coprime} to the previous lemma we get the following corollary.

\begin{corollary}
If we start with $n$ chips, the final state for the $a$-$a$ firing has the remainder of $n$ divided by $2a$ chips at the origin. It also has $\lfloor \frac{n}{2a} \rfloor$ consecutive vertices with $a$ chips on the each side of the origin.
\end{corollary}

From now on we assume that $a$ and $b$ are coprime. We also assume that $a < b$, as the case $a > b$ can be described by symmetry.

\section{$2$-$3$ Firing.}\label{sec:23firing}

As an example, let us consider a 2-3 firing. Here is the first several final states $\phi(n)$ for $n$ starting chips, where $n$ ranges from 0 to 28: 0., 1., 2., 3., 4., 20.3, 21.3, 22.3, 23.3, 24.3, 42.13, 43.13, 44.13, 213.43, 214.43, 232.413, 233.413, 234.413, 422.2413, 423.2413, 424.2413, 442.2243, 443.2243, 444.2243, 2332.222413, 2333.222413, 2334.222413, 4222.2222243.

Let us look separately at the left and right parts. Here is the sequence $\phi(n)^L$ of left parts:
\[0,\ 1,\ 2,\ 3,\ 4,\ 20,\ 21,\ 22,\ 23,\ 24,\ 42,\ 43,\ 44,\ 213,\ 214,\ 232,\]
\[233,\ 234,\ 422,\ 423,\ 424,\ 442,\ 443,\ 444,\ 2332,\ 2333,\ 2334,\ 4222.\]

For $n \geq 13$, the transition from $\phi(n)^L$ to $\phi(n+1)^L$ can be described by adding 1 to the last digit and have some explosions afterwards. The explosions are similar to exploding dots in base $\frac{3}{2}$. If a place has more than 4 dots, then 3 dots explode adding two dots to the place on the left. This is the same rule as for the base $\frac{3}{2}$ except there is a new lower limit when the explosions are allowed.

For example, $\phi(17)^L = 234$. We add 1 more dot to the unit's place to get 235. Then 3 dots at the unit's place explode producing 252. Then we have another explosion, where we have 5, getting to 422. The explosions stop as we do not have any digit exceeding 4. Thus, $\phi(18)^L = 422$. We discuss the left part for $a$-$b$ chip-firing in Section~\ref{sec:leftside}.

We now look at the right parts $\phi(n)^R$ starting from $n=13$:
\[.43,\ .43,\ .413,\ .413,\ .413,\ .2413,\ .2413,\ .2413, \]
\[ .2243,\ .2243,\ .2243,\ .222413,\ .222413,\ .222413,\ .2222243.\]

We used a computer to calculate the final state for $n < 100$. We can see that starting from $n=13$ the right part is either $2_k43$ or $2_k413$ for some $k$. The right parts can be grouped into identical triplets:
\[\phi(3k)^R = \phi(3k+1)^R = \phi(3k+2)^R,\]
for $n > 4$.

We want to describe the transition from $\phi(3k+2)^R$ to $\phi(3k+3)^R$ for $n>4$. According to our computation, $\phi(3k+2)^L$ ends in 4.  Let $j$ denote the number of digits in the longest ending piece of $\phi(3k+2)^L$ that consists only of threes or fours. Then for each count in $j$ number 413 is replaced by 243, and 43 is replaced by 413. Coincidentally, but not surprisingly, $j$ is the number of carries (explosions) needed when transitioning from  $\phi(3k+2)^L$ to $\phi(3k+3)^L$.

For example, $\phi(17) = 234.413$. In this case $j = 2$. Thus we do the replacements twice. First, we replace 413 with 243. Then we replace 43 with 413, changing 243 to 2413. Thus, $\phi(18)^R = .2413$. We discuss the right part for $a$-$b$ chip-firing in Section~\ref{sec:rightside}.

We can observe another important property for $n > 4$. Each part evaluated in base $\frac{3}{2}$ stabilizes:
\[ (\phi(n)^R)_{\frac{3}{2}} = 2 \quad \textrm{ and } \quad (\phi(n)^L)_{\frac{3}{2}} = n-2.\] 
We discuss the analogous property for $a$-$b$ chip-firing in Section~\ref{sec:value}.

Before describing the final state we study what happens to the right part after the origin fires.

\section{What one origin firing does to the right part}\label{sec:originfiring}

In this section we look what happens to the right part after the origin fires. After the origin fires, we continue with all possible firings in the right part, while ignoring the left part, we call this process \textit{settling} the right part. We call the right part after it can no longer fire the \textit{settlement}. We denote the sequence of settlements as $\xi_k$. For example, the sequence $\xi_k$ for the 2-3 chip-firing game starts as follows:
\[.,\ .3,\ .13,\ .43,\ .413,\ .243,\ .2413,\ .2243,\ .22413, \ldots.\]

The connection between the settlements and the final right parts $\phi(n)^R$ is as follows: the right parts are settlements and the settlement index does not decrease when the number of chips increase. Formally, for any $n$ there exists $i$,  such that $\phi(n)^R = \xi_i$, and $\phi(n+1)^R = \xi_j$, where $j \geq i$.

Now we describe how to calculate $\xi_{k+1}$ from $\xi_k$. Suppose $\xi_k = .s_1s_2\ldots s_\ell$, and $\xi_{k+1} = .s_1's_2'\ldots s_\ell's_{\ell+1}'$, where we have $s_{\ell+1} = 0$ and we allow $s_{\ell+1}' = 0$. Let $j$ be the smallest $i$ such that $s_i < a$. If $s_i \geq a$, for $1 \leq i \leq \ell$, we assume that $j = \ell+1$. The following lemma describes the rule how we can get $\xi_{k+1}$ from $\xi_k$.

\begin{lemma}\label{lemma:transition}
In an $a$-$b$ chip-firing game, the transition from $\xi_k$ to $\xi_{k+1}$ can be described as follows. If $j = 1$, then $s_1' = s_1 + b$, and $s_i' = s_i$, for $i > 1$. If $j > 1$, then $s_{j-1}' = s_{j-1}-a$, and $s_{j}' = s_{j}'+b$, otherwise $s_i' = s_i$.
\end{lemma}

\begin{proof}
If $j = 1$, then $s_1 < a$. After origin fires, the value $s_1$ is changed to $s_1+b$. The new value is less than $a+b$, so there are no more firings. 

If $j > 1$, then the vertices 1 through $j$ fire exactly once in order and here is the sequence of the firing results, where we use bold for the number that is bigger than $a+b-1$:
\[.\mathbf{(s_1+b)}s_2\ldots s_\ell \quad  \rightarrow  \quad   .(s_1-a)\mathbf{(s_2+b)}s_3\ldots s_\ell \quad  \rightarrow \]
\[.s_1(s_2-a)\mathbf{(s_3+b)}s_4\ldots s_\ell \quad  \rightarrow  \quad  .s_1s_2(s_3-a)\mathbf{(s_4+b)}s_5\ldots s_\ell \quad  \rightarrow \ldots\]
\[.s_1\ldots s_{j-3}(s_{j-2}-a)\mathbf{(s_{j-1}+b)}s_{j}\ldots s_\ell \quad  \rightarrow \]
\[.s_1\ldots s_{j-2}(s_{j-1}-a)(s_{j}+b)s_{j+1}\ldots s_\ell.\]
At this point all firings stop as every vertex on the right side has less than $a+b$ chips.
\end{proof}

We call a settlement \textit{dormant} if the value at the origout is less than $a$. Such settlements are important, as after the origin fires at them, the origout does not fire back. The proof of Lemma~\ref{lemma:transition} also implies the following corollary.

\begin{corollary}
After the origin fires once, during the settling period, the origout fires 0 or 1 times, depending on whether or not the right part is a dormant settlement.
\end{corollary}

\section{Settlements}\label{sec:settlements}

The value $\lceil \frac{a}{b-a}\rceil$ plays an important role in the coming formulae and proofs. We denote it by $c$. 

\begin{lemma}
If for a positive integer $m$, we have $m \leq c$, then 
\[mb -(m-1)a < a+b \leq (m+1)b - (m-1)a.\]
Also $c(b-a) < a+b$.
\end{lemma}

\begin{proof}
First $(m+1)b - (m-1)a \geq 2b > a+b$ independently of the value of positive integer $m$. If $m \leq c$, then $m -1 < \frac{a}{b-a}$. Hence, $mb -(m-1)a = (m-1)(b-a) + b < a+b$. The final inequality follows from the first inequality: $c(b-a) < cb - (c-1)a < a+b$.
\end{proof}

This means that for $m \leq c$, the values $mb -(m-1)a$ and $mb -ma$ do not fire, while the value $(m+1)b - (m-1)a$ fires.

Now we calculate explicitly the sequence $\xi_k$ of settlements. For starters, we have
\[\xi_0 = .0, \quad \xi_1 = .b, \quad \xi_2 = .(b-a)b.\]

For the following lemma we denote triangular numbers as $T_i$, where $T_i = \frac{i(i+1)}{2}$, and tetrahedral numbers as $Te_i$, where
\[Te_{n}=\sum _{k=1}^{n}T_{k}=\sum _{k=1}^{n}{\frac {k(k+1)}{2}}=\sum _{k=1}^{n}\left(\sum _{i=1}^{k}i\right) = \binom{n+2}{3}.\]

\begin{lemma}\label{lemma:tetrahedral}
After the origin fires $Te_c+1$ times, we get the settlement
\[\xi_{Te_c+1} = .(cb-ca)(cb - (c-1)a)\ldots (ib - (i-1)a) \ldots(3b-2a)(2b-a)b,\]
where $cb - ca > a$.
Moreover the number of dormant settlements $\xi_i$, where $i < Te_c+1$ is $c$.
\end{lemma}

\begin{proof}
The proof is by induction. If $c = 1$, then $\frac{a}{b-a} < 1$, that is $b- a > a$. We see that $\xi_2 = .(b-a)b$, where $b-a > a$ as required. And we have only one dormant settlement, namely $\xi_0 = .0$, in the range.

If $c >1$, then $b-a < a$ and 
\[\xi_3 = .(2b-a)b, \quad \xi_4 = .(2b-a)(b-a)b, \quad \xi_5 = .(2b-2a)(2b-a)b.\]

If $c =2$, then $\xi_5$ matches the formula in the statement of the lemma. Also, one more dormant settlement, namely $\xi_2 = .(b-a)b$, is added to the range.

Now we use induction in $k$, for $k < c$. By our assumption
\[\xi_{Te_k+1} = .(kb-ka)(kb - (k-1)a)\ldots (ib - (i-1)a) \ldots(3b-2a)(2b-a)b.\]
Also, we assume that the number of dormant settlements in the range $[1,Te_k]$ is $k-1$. 

As $k < c$, we have $kb-ka < a$, thus the next settlement is 
\[.((k+1)b-ka)(kb - (k-1)a)\ldots (ib - (i-1)a) \ldots(3b-2a)(2b-a)b.\]
Notice that all the digits in this settlement are greater than $a$. Thus, applying Lemma~\ref{lemma:transition}, the next settlement is
\[.((k+1)b-ka)(kb - (k-1)a)\ldots (ib - (i-1)a) \ldots(3b-2a)(2b-a)(b-a)b.\]
Now the first digit that is less than $a$ is $b-a$, thus, again, by Lemma~\ref{lemma:transition}, the next settlement is
\[.((k+1)b-ka)(kb - (k-1)a)\ldots (ib - (i-1)a) \ldots(3b-2a)(2b-2a)(2b-a)b.\]
After several more rounds we get (the term of the form $(mb-ma)$ keeps moving to the left)
\begin{multline*}
.((k+1)b-ka)\ldots \\
((m+1)b - ma)(mb-ma)(mb - (m-1)a) \ldots \\
 (3b-2a)(2b-a)b,
\end{multline*}
until we reach
\[.((k+1)b-(k+1)a)((k+1)b - ka)\ldots(3b-2a)(2b-a)b,\]
as desired.
Notice that we added one new dormant settlement, namely $\xi_{Te_k+1}$, to the range in question.
\end{proof}

Let us denote the string $(cb - (c-1)a)\ldots (ib - (i-1)a) \ldots(3b-2a)(2b-a)b$ as $\delta_0$, the string $(cb-(c-1)a)((c-1)b - (c-2)a)\ldots (3b-2a)(2b-a)(b-a)b$ as $\delta_1$ and so on: the string 
\begin{multline*}(cb-(c-1)a)((c-1)b - (c-2)a)\ldots \\
((m+1)b - ma)(mb-ma)(mb - (m-1)a) \ldots \\
(3b-2a)(2b-a)b
\end{multline*}
as $\delta_m$, and lastly the string
\[(cb-(c-1)a)((c-1)b - (c-1)a)((c-1)b - (c-2)a) \ldots (3b-2a)(2b-a)b\]
as $\delta_{c-1}$.

Recall, that we denoted a word that consists $k$ instances of a digit $d$ as $d_k$, Now we are ready to describe the eventual behavior of the settlements.

\begin{theorem}
The following formula describes the settlements starting from index ${Te_c+1}$. If $n = {Te_c+1} + pc + q$, where $q <c$, then
\[\xi_n = .(cb - ca)_{p+1}\delta_q.\]
\end{theorem}

\begin{proof}
The proof follows from the fact that $cb - ca >a$ and the calculations made in the previous proof.
\end{proof}

\begin{corollary}
The highest index of the dormant settlements is $Te_{c-1}+1$.
\end{corollary}

\textbf{Example.} Consider the 3-4 firing. In this case $c = \lceil \frac{3}{4-3}\rceil = 3.$ That means, the repeated digit is $cb-ca = 3$. The digits on the very right cycle through $\delta_0 = (3b-2a)(2b-a)b$, $\delta_1 = (3b-2a)(2b-a)(b-a)b$, and $\delta_2 = (3b-2a)(2b-2a)(2b-a)b$. Explicitly, they cycle through $\delta_0 = 654$, $\delta_1 = 6514$, and $\delta_2 = 6254$.

\textbf{Example.} Suppose $(b-a) \geq a$, then $c=1$ and the sequence of settlements is described by
\[\xi_k = .(b-a)_{k-1}b,\]
for $k > 0$.

Table~\ref{table:origout} shows eventual origout values in the final states for large values of $n$ for some values of $a$ and $b$.

\begin{table}[h]
\begin{center}
  \begin{tabular}{| c | c | c |}
\hline
  & $c$ & origout value \\
\hline
$a \leq b/2$  & $1$ & $b-a$ \\
$b=(a+1)$ & $a$ & $a$ \\
$a =2k-1 = b-2$ & $k$ & $2k = a+1$ \\
\hline
  \end{tabular}
\end{center}
  \caption{Eventual origout values.}\label{table:origout}
\end{table}

We call the number of chips $B$ \textit{balanced}, if $B$ is the smallest integer such that $\phi(B)^R = \xi_k$, for $k > Te_{c-1}+1$. Starting from $B$ the origout value of the final state is always at least $a$, that is, every time the origin fires, the origout fires back.

\section{The value of the final left/right part in base $\frac{b}{a}$}\label{sec:value}

In this section we discuss $(\phi(n)^L)_{\frac{b}{a}}$ and $(\phi(n)^R)_{\frac{b}{a}}$, that is, the values of the left and right parts of the final state evaluated in base $\frac{b}{a}$.

Let us denote the total number of firings of the origin needed to get to the final state for $n$ chips as $f_0$ and for the origout as $f_1$. We denote the corresponding values for $n+1$ as $f_0'$ and $f_1'$. Recall that by Corollary~\ref{cor:uniequelydefined} these numbers are well-defined. We have 
\[f_0' \geq f_0 \quad \textrm{ and } \quad f_1' \geq f_1.\]
Indeed, we can start with $n$ chips, get to the final state, then add 1 chip to the origin. After this we might need more firings.

\begin{lemma}\label{lemma:firingdifference}
We have
\[f_0' - f_1' \geq f_0 - f_1.\]
Moreover, for any $n \geq B$
\[f_0' - f_1' =c = \lceil \frac{a}{b-a}\rceil.\]
\end{lemma}

\begin{proof}
As follows from the previous section, when the origin fires the origout fires back once if the settlement on the right is not dormant, and does not fire back if it is dormant. For $n \geq B$, the right part is not a dormant settlement. That means that whenever the origin fires, the origout fires back once. Hence, the difference between these two firing numbers stays constant. From Lemma~\ref{lemma:tetrahedral} we know that there are $c$ dormant settlements among the first $Te_{c-1}+2$. That means this constant is exactly $c$.
\end{proof}

We now evaluate the sides in base $\frac{b}{a}$ for  $n$ that is large enough. But first, recall the definition of the ssate polynomials from the beginning of Section~\ref{sec:series}.

\begin{theorem}\label{thm:stable}
For any $n \geq B$, we have
\[S_{\phi(n)^R}\left( \frac{b}{a} \right) = ac\]
and
\[S_{\phi(n)^L}\left( \frac{b}{a} \right) = n -ac,\]
where $c = \lceil \frac{a}{b-a}\rceil$.
\end{theorem}

\begin{proof}
From Proposition~\ref{prop:valueAndFiring} we know that $(\phi(n)^R)_{\frac{b}{a}} = S_{\phi(n)^R}\left( \frac{a}{b} \right) = a(f_0 - f_1)$ and $(\phi(n)^L)_{\frac{b}{a}} =S_{\phi(n+1)^L}\left( \frac{a}{b} \right) = n -a(f_0 - f_1)$. The result follows from Lemma~\ref{lemma:firingdifference}.
\end{proof}

\textbf{Example.} Consider the 3-4 firing. In this case $c = \lceil \frac{3}{4-3}\rceil = 3.$ That means, the repeated digit in the right part is $cb-ca = 3$. The value of the right part evaluated in base $\frac{4}{3}$ eventually becomes $ac = 9$.

Table~\ref{table:valuerightside} shows $(\phi(n)^R)_{\frac{b}{a}}$ for large values of $n$ for some values of $a$ and $b$.

\begin{table}[h]
\begin{center}
  \begin{tabular}{| c | c | c |}
\hline
  & $c$ &  the eventual value of $(\phi(n)^R)_{\frac{b}{a}}$\\
\hline
$a \leq b/2$  & $1$ &  $a$ \\
$b=(a+1)$ & $a$ &  $a^2$ \\
$a =2k-1 = b-2$ &  $2k = a+1$ & $(2k-1)k$ \\
\hline
  \end{tabular}
\end{center}
  \caption{The eventual value of $(\phi(n)^R)_{\frac{b}{a}}$.}\label{table:valuerightside}
\end{table}

\section{The Left Part}\label{sec:leftside}

We showed in the previous section that starting from $B$ chips whenever the origin fires, the origout fires too. We can say: when the origin fires it sends $a$ chips to the left and receives $a$ chips. In other words: as soon as the origin has $a+b$ chips, $a+b$ chips explode, and $a$ chips move to the left and another set of $a$ chips stays in place. Yet in other words: as soon as the origin has $a+b$ chips, $b$ chips explode, and $a$ chips are added to the left. This is similar to the behavior of the fractional base $\frac{b}{a}$ except in the latter case we need only $b$ chips in order to explode. We call the new game the \textit{elevated} $a$-$b$ chip-firing game.

\begin{lemma}\label{lemma:elevated}
Consider $n_0 > B$, where $\phi(n_0)^L = s_{-k}\ldots s_{-2}s_{-1}s_0$. Suppose $i$ is the smallest integer such that $s_{-i} <a$. Then, for any $n > n_0$, the transition from $\phi(n)^L$ to $\phi(n+1)^L$ on the vertices in the range $[1-i,0]$ is the same as adding one chip at the origin in the elevated $a$-$b$ chip-firing game.
\end{lemma}

\begin{proof}
We can represent the last $i$ digits of $\phi(n)^L$ as $Xd_kd_{k-1}\ldots d_2 d_1 d_0$, where string $X$ ends in a digit not exceeding $b$, and $b \leq d_i < a+b$ for any $i$. If $d_0 < a+b-1$, then we add one chip to the origin and no firing happens. 
If $d_0 = a+b-1$, and we add one chip to the origin, the origin fires:
\[Xd_kd_{k-1}\ldots d_2 (d_1+a) a.\]
Then the digit next to the origin fires, causing the origin to fire again:
\[Xd_kd_{k-1}\ldots (d_2+a) (d_1-b) (a+b) \quad Xd_kd_{k-1}\ldots (d_2+a) (d_1-b+a) a.\]
Now the hundreds' digit fires:
\[Xd_kd_{k-1}\ldots (d_3+a) (d_2-b)(d_1+a) a,\]
causing the tens' digit and the origin to fire:
\begin{multline*}
Xd_kd_{k-1}\ldots (d_3+a) (d_2-b+a)(d_1-b) (a+b) \\
 Xd_kd_{k-1}\ldots (d_3+a) (d_2-b+a)(d_1-b+a) a,
\end{multline*}
Then the thousands' digit fires, causing all the consecutive digits to fire ending in:
\[Xd_kd_{k-1}\ldots (d_4+a) (d_3-b+a)(d_2-b+a)(d_1-b+a) a.\]
After we reach $d_k$ and it fires, it causes all the digits to the right to fire ending in:
\[(X+a)(d_k-b+a)(d_{k-1}-b+a)\ldots (d_3-b+a)(d_2-b+a)(d_1-b+a) a.\]
The final result is exactly the same as adding a chip at the origin in the elevated $a$-$b$ chip-firing game.
\end{proof}

Now we are ready for our theorem about the left part.

\begin{theorem}
There exists $N > B$, so that for any $n > N$, the final state $\phi(n)^L = s_{-k}\ldots s_{-2}s_{-1}s_0$ has all the digits greater or equal to $a$. In addition, the transition from $\phi(n)^L$ to $\phi(n+1)^L$ is the same as adding one chip at the origin in the elevated $a$-$b$ chip-firing game.
\end{theorem}

\begin{proof}
We know that for $n > B$ each time the origin fires, the origout fires too. That means the origin is at least $a$. 

Consider $n_0 > B$, where $\phi(n_0)^L = s_{-k}\ldots s_{-2}s_{-1}s_0$. Suppose $i$ is the smallest integer such that $s_{-i} <a$. Then from Lemma~\ref{lemma:elevated} for any $n > n_0$ the last $i$ digits of $\phi(n)^L$ are at least $a$. For some $n_1 \geq n_0$, the vertex $1-i$ has to fire, making the last $i+1$ digits of $\phi(n_1)^L$ at least $a$, and not changing the total number of digits. Thus, for any $n > n_1$, the last $i+1$ digits of $\phi(n)^L$ are at least $a$. We proceed until all the digits of $\phi(n_2)^L$ for some number $n_2$ are at least $a$. Also, if a new digit appears to the left it is $a$. By Lemma~\ref{lemma:elevated} all the digits stay at least $a$ for any $n > n_2$.

This means, for $n > n_2$, the transition from $\phi(n)^L$ to $\phi(n+1)^L$ is the same as adding a chip to the origin in the elevated $a$-$b$ chip-firing game.
\end{proof}

We denote by $H \geq B$ the smallest integer $n$ starting from which all the digits in the left part are at least $a$ and the chip-firing on the left is the same as adding a chip to the origin in the elevated $a$-$b$ chip-firing game.

\section{The Right Part}\label{sec:rightside}

We know that if $\phi(n)^R = \xi_k$, then $\phi(n+1)^R = \xi_{k+i}$, where $i \geq 0$. The following theorem described the value of $i$ for large $n$.

\begin{theorem}
For $n > H$, if $\phi(n)^R = \xi_k$ and $i$ is the number of digits in the largest suffix of $\phi(n)^L$ that contain digits that are greater or equal to $b$, then $\phi(n+1)^R = \xi_{k+i}$.
\end{theorem} 

\begin{proof}
Consider the transition from $\phi(n)^R$ to $\phi(n+1)^R$. As we saw in the proof of Lemma~\ref{lemma:elevated} the number of times the origin fires in this transition is exactly $i$.
\end{proof}

\section{$1$-$b$ Firing Example}\label{sec:1bfiring}

The formulae are particularly simple for 1-$b$ firing for $b > 1$. This case corresponds to the first line in the Tables~\ref{table:origout} and~\ref{table:valuerightside}, where $c=1$. 

The repeated digit to the right of the radix is $b-1$, and the eventual value of the right part evaluated in base $b$ is $1$. Correspondingly, for $n > b$, the left part of 1-$b$ chip-firing interpreted in base $b$ is equal $n-1$. The value of the right part stabilizes after the first origin firing, that is, after $n = 1+b$. 

We have the explicit formula for settlements:
\[\xi_k = (b-1)_{k-1}b.\]

The number of digits $b-1$ in $\phi(n)^R$ for $n > b$ is 
\[\sum_{i=1}^{n-b-2}\nu_b(i),\]
where $\nu_b(x)$ the $b$-adic valuation of $x$.  That is, $\nu_b(x)=k$, where $k$ is the largest power of $b$ that divides $x$.

Now we wish to describe the left part.
Consider a sequence $R(b)$ that depends on an integer parameter $b$. Each element of $R(b)$ is a finite integer string consisting of digits 1 through $b$. The sequence is arranged in an increasing order by value if evaluated in any base greater than 2. In particular, $R(b)_1=1$, $R(b)_b = b$, and $R(b)_{b+1} = 11$. Here is, for example, sequence $R(2)$:
\[1,\ 2,\ 11,\ 12,\ 21,\ 22,\ 111,\ 112, \ldots.\]

We need this sequence to help us describe how the left part of the final state of 1-$b$ chip-firing changes from one number to the next.

Consider the final state for $n$ equaling $b$, $b+1$ and $b+2$ in 1-$b$ firing:
\[\phi(b) = b \quad \phi(b+1) = 10.b  \quad \phi(b+2) = 11.b.\]
We see that $\phi(b+2)^L = R(b)$. After that the left part before the radix follows sequence $R(b)$, that is for $n > b+1$, we have $\phi(n)^L = R(n-2)$.

The left part is particularly easy to describe for $b=2$. That is, for 1-2 firing. Here is a description of the left part or the  final state for $n > 3$.

To get the left part we represent $n$ in binary, remove the first digit, and increase all the other digits up by 1. For example, 21 in binary is 10101. Remove the first digit to get 0101, then shift by 1 to get 1212.

\section{Acknowledgements}

We are grateful to Prof.~James Propp for suggesting this project to us. We also want to thank PRIMES STEP for giving us the opportunity to do this research. We are grateful to the anonymous reviewers for useful comments.

\end{document}